\documentclass{amsart}

\author{Michel Dekking and  Derong Kong} 

\usepackage{amsmath}
\usepackage{amssymb}
\usepackage{amsopn}
\usepackage{epsfig}
\usepackage{amsfonts}
\usepackage{latexsym}
\usepackage{graphicx}

\newcommand{\prob}[1]{\ensuremath{{\rm P}_{\!\In}\left( #1 \right)}}
\newcommand{\prc}[2]{\ensuremath{{\rm P}_{\!\In}( #1 \, |\, #2)}}
\newcommand{\expec}[1]{\ensuremath{{\rm E}_{\In}\!\left[#1\right]}}
\newcommand{\expectau}[1]{\ensuremath{{\rm E}_{\bar{\tau}}\!\left[#1\right]}}
\newcommand{\expecab}[1]{\ensuremath{{\rm E}_{_{\scriptstyle\ab}}\!\left[#1\right]}}
\newcommand{\expecfr}[1]{\ensuremath{{\rm E}_{_{\scriptstyle\f}}\!\left[#1\right]}}
\newcommand{\varab}[1]{\ensuremath{{\rm Var}_{_{\scriptstyle\ab}}\!\left( #1 \right)}}
\newcommand{\varfr}[1]{\ensuremath{{\rm Var}_{_{\scriptstyle\f}}\!\left( #1 \right)}}
\newcommand{\var}[1]{\ensuremath{{\rm Var}_{\In}\!\left( #1 \right)}}
\newcommand{\In}{\nu}       
\newcommand{\st}{\pi}    
\newcommand{\ei}{\gamma}   
\newcommand{\ab}{{\textsc{F}}}          
\newcommand{\f}{{\textsc{S}}}          
\newcommand{\ea}{\varepsilon_{\ab}}          
\newcommand{\ef}{\varepsilon_{\f}} 
\newcommand{\fab}{f^{\ab}}
\newcommand{\ff}{f^{\f}}
\newcommand{\jab}{\hat{f}^{\ab}}
\newcommand{\jf}{\hat{f}^{\f}}
\newcommand{\kab}{K_n^{\ab}}
\newcommand{\kf}{K_n^{\f}}

\theoremstyle{plain}
\newtheorem{thm}{Theorem}[section]

\newtheorem{lem}{Lemma}[section]
\newtheorem{prop}{Proposition}[section]

\newtheorem{ex}{Example}[section]

\begin{document}

\title{Multimodality of the Markov binomial distribution} 

\address{Delft
University of Technology{Michel Dekking} 
3TU Applied Mathematics Institute and Delft
University of Technology, Faculty EWI, P.O.~Box 5031, 2600 GA Delft,
The Netherlands.  } 
\email{F.M.Dekking@math.tudelft.nl,\quad D.Kong@tudelft.nl}

\maketitle

\begin{abstract}
We study the shape of the probability mass function of the Markov binomial distribution,
 and give necessary and sufficient conditions for the probability mass function to be unimodal, bimodal or trimodal.
 These are useful to analyze the double-peaking results from a PDE reactive transport model from
 the engineering literature. Moreover, we give a closed form expression for the variance of the Markov
binomial distribution, and expressions for the mean and the variance
conditioned on the state at time $n$.
\end{abstract}

\medskip


\smallskip

{\it AMS Maths Subject Classification numbers:}  {60J10} {60J20} 
\section{Introduction}

 The Markov binomial distribution occurs in diverse applications. Examples are weather forecasting,
 stock market trends, DNA matching, quality control (cf.~\cite{Omey}), and biometrics (cf.~\cite{Daugman},
 see also \cite{iris}). In 1924  Markov \cite{Markov}
 showed that under certain conditions a Markov binomial distribution is asymptotically normally distributed. Later in 1953
 Dobru\v{s}in \cite{MR0058150} studied some other limit distributions of a Markov binomial
 distribution.  In 1960 Edwards \cite{Edwards} rediscovered the Markov
 binomial distribution  in connection with work on the human sex ratio. More recently many authors studied its
 distribution and moments
(cf.~\cite{ MR0109363, MR0282416, Viveros}) and its approximations by
compound Poisson distributions and   binomial distributions
(cf.~\cite{MR2327530, MR2485024, MR633240}).

Our interest in the possible lack of unimodality of the Markov binomial distribution arose
 from the paper \cite{Mich} where the authors deduce from simulations a somewhat surprising behaviour
 of double peaking in the concentration of the aqueous part of a solute undergoing kinetic adsorption and moving by advection and dispersion. In our paper \cite{SKTM} we will explain this behaviour rigorously from the multimodality properties that we derive in the present paper.

 Let $\{Y_k, k\ge 1\}$ be a Markov chain on the two states $\{\f,\ab\}$ with initial distribution $\In=(\In_\f, \In_\ab)$ and transition matrix
 \begin{equation}\label{transition matrix}
P=\left[ {\begin{array}{*{20}c}
   {P(\f,\f) } & {P(\f,\ab) }  \\
   {P(\ab,\f) } & {P(\ab,\ab) }  \\
\end{array}}  \right] =\left[ {\begin{array}{*{20}c}
   {1 - a} & {a}  \\
   {b} & {1-b}  \\
\end{array}} \right],
\end{equation}
where we assume $0<a, b< 1$ throughout the paper.
  The \emph{Markov binomial distribution} (MBD) is
 defined for $n\ge 1$ as the distribution of the random variable which counts the number of successes in $n$ experiments with the two outcomes success and failure:
 $$K_n=\sum_{k=1}^n {\bf 1}_{\{Y_k=\f\}}.$$

 We say $K_n$ is a ${\mathit Bin}(n,a,b,\In)$ distributed random variable. Clearly the MBD generalizes the binomial distribution, where $a+b=1$
  and $(\In_\f,\In_\ab)=(b,a)$.

In Section~\ref{var} we will give an explicit formula for the variance of a MBD. This was not given in \cite{Viveros},
 and only implicitly in \cite{MR0282416, Omey}. By introducing the notion of `excentricity' we can write down tractable formulas for the expectation and the variance of a MBD.
For the  application to the reactive transport model we need a bit
more, namely the variances conditioned on the state of the chain at
time $n$.
 Expressions for these formulas will be computed in Section~\ref{varcond}.

 In Section~\ref{PMF} we will give a closed formula for the probability mass function $f_n$ of $K_n$, and we study its shape. The probability mass
 function $f_n$ was implicitly given in \cite{MR0109363, MR0282416, Viveros}, but the closed formula presented here is helpful to study its
 shape. Surprisingly, the shape can be unimodal, bimodal and trimodal.  We show in particular  that when $a+b\ge 1$ the probability mass function of $K_n$ is unimodal, and that the probability mass function of $K_n$ restricted to the interval $[1,n-1]$ is always unimodal.

In Section~\ref{sec:condmass} we give formulas for the  probability
mass functions of $K_n$, \emph{conditional} on the state at time $n$.
Here again our interest arises from the fact that in the reactive
transport model of \cite{Mich} the authors consider the behaviour of
the concentration of the aqueous part of a solute, which corresponds
to conditioning at the state of the chain at time $n$ (aqueous
$\sim$ success, adsorbed $\sim$ failure).

\section{The variance of the Markov binomial distribution}\label{var}

Let $(\st_\f,\st_\ab)$ be the stationary distribution of the chain
$\{Y_k,k\ge 1\}$. We have
$$\st_\f=\frac{b}{a+b}, \quad \st_\ab=\frac{a}{a+b}.$$
In fact, diagonalizing $P$ yields for $n=0,1,2\dots$
\begin{equation}\label{P^n}
P^n=\left[ {\begin{array}{*{20}c}
   {\st_\f } & {\st_\ab }  \\
   {\st_\f } & {\st_\ab }  \\
\end{array}}  \right]+\ei^n\left[ {\begin{array}{*{20}c}
  ~\;\st_\ab   & {-\st_\ab }  \\
   {-\st_\f } &~\;\st_\f   \\
\end{array}}  \right],
\end{equation}
where $\ei=1-a-b$ is the second largest eigenvalue of $P$.  Note that
for $1\le k\le n$,
\begin{equation*}
\prob{Y_k=\f}=\In_\f P^{k-1}(\f,\f)+\In_\ab
P^{k-1}(\ab,\f)=\st_\f\big(1-(1-\frac{\In_\f}{\st_\f})\ei^{k-1}\big),
\end{equation*}
and similarly,
$$\prob{Y_k=\ab}=\In_\f P^{k-1}(\f,\ab)+\In_\ab
P^{k-1}(\ab,\ab)=\st_\ab\big(1-(1-\frac{\In_\ab}{\st_\ab})\ei^{k-1}\big).$$
It appears thus useful to define the \emph{excentricities} $\ef$ and
$\ea$ of an initial distribution $\In$ by
$$ \varepsilon_\tau:=\varepsilon_\tau(\In)=1-\frac{\In_\tau}{\st_\tau},\quad\mbox{for}\quad \tau\in\{\f, \ab\}.$$
Both quantities measure the deviation of the initial distribution $\In$ from the stationary
distribution $\st$.
Using them we can rewrite $\prob{Y_k=\f}$ and $\prob{Y_k=\ab}$ as
\begin{equation}\label{P(Y_k)}
\prob{Y_k=\f}=\st_\f(1-\ef\,\ei^{k-1}),\quad
\prob{Y_k=\ab}=\st_\ab(1-\ea\,\ei^{k-1}).
\end{equation}

Moreover, the expectation of $K_n$ is given by (note that $\ei<1$ since $a+b>0$)
\begin{equation}\label{expect}
\expec{K_n}=\sum_{k=1}^n\expec{\textbf{1}_{\{Y_k=\f\}}}=\sum_{k=1}^n\prob{Y_k=\f}=\st_\f\Big(n-\ef\frac{1-\ei^n}{1-\ei}\Big).
\end{equation}
The expectation of $K_n$ is particularly simple if we start in the
equilibrium distribution, since in this case $\ef=0$.

Obtaining $\var{K_n}$ is more involved, because of correlations.
\begin{prop}\label{prop:var_K_n}
  For any ${\mathit Bin}(n,a,b,\In)$ distributed random variable $K_n$, we have
  \begin{equation*}
  \begin{split}
  \var{K_n}=&~\st_\f\left\{n\,\frac{\st_\ab(1+\ei)}{1-\ei} +\frac{\ei(\ef(\st_\f-\st_\ab)-2\st_\ab)
  -\ef(\st_\ab-\In_\f)}{(1-\ei)^2}+n \ei^n\,\frac{2\ef(\st_\ab-\st_\f)}{1-\ei}\right.\\
  &\quad\quad\quad+\left.\ei^n\Big(\frac{\ef(\st_\f-\st_\ab)}{1-\ei}
  +2\,\frac{\ei\,\st_\ab+\ef(\st_\ab-\In_\f)}{(1-\ei)^2}\Big)-\ei^{2n}\frac{\st_\f\,\ef^2}{(1-\ei)^2}
   \right\}.
   \end{split}
\end{equation*}
\end{prop}
\begin{proof}
  Since $\var{K_n}=\expec{K_n^2}-(\expec{K_n})^2$, using
  (\ref{expect}) it suffices to calculate
  \begin{equation*}
  \begin{split}
  \expec{K_n^2}=&~\expec{\Big(\sum_{k=1}^n{\bf 1}_{\{Y_k=\f\}}\Big)^2}=\sum_{k=1}^n \prob{Y_k=\f}+2\!\sum_{1\le i<j\le n}\prob{Y_i=\f, Y_j=\f}\\
  =&~\expec{K_n}+2\!\sum_{1\le i<j\le n}\prob{Y_i=\f, Y_j=\f}.
  \end{split}
  \end{equation*}
  Thus we only need to calculate
  \begin{equation*}
  \begin{split}
    \prob{Y_i=\f,Y_j=\f}=&~\prc{Y_j=\f}{Y_i=\f}\prob{Y_i=\f}=(\st_\f+\st_\ab\ei^{j-i})\st_\f(1-\ef\ei^{i-1})\\
    =&~\st_\f(\st_\f+\st_\ab\ei^{j-i}-\ef\,\st_\f\ei^{i-1}-\ef\,\st_\ab\ei^{j-1}),
  \end{split}
  \end{equation*}
  using (\ref{P^n}) and (\ref{P(Y_k)}). Performing the four summations we
  obtain that
  \begin{equation*}
  \begin{split}
    &~2\!\sum_{1\le i<j\le n}\prob{Y_i=\f,Y_j=\f}\\
    =&~2\st_\f\left\{\st_\f\frac{n(n-1)}{2}+\st_\ab\ei\Big(\frac{n}{1-\ei}-\frac{1-\ei^n}{(1-\ei)^2}\Big)
    -\ef\st_\f\Big(\frac{n}{1-\ei}-\frac{1-\ei^n}{(1-\ei)^2}\Big)\right.\\
    &\hspace{7.5cm}-\left.\ef\st_\ab\Big(\frac{-n\ei^n}{1-\ei}+\frac{\ei(1-\ei^n)}{(1-\ei)^2}\Big)\right\}\\
    =&~\st_\f\left\{n(n-1)\st_\f +2 n\,\frac{\st_\ab\ei-\ef\,\st_\f}{1-\ei}
    +2n\ei^n\frac{\ef\,\st_\ab}{1-\ei}+2(1-\ei^n)\frac{\ef\,\st_\f-\st_\ab\ei(1+\ef)}{(1-\ei)^2}
    \right\},
    \end{split}
  \end{equation*}
  which, combined with (\ref{expect}), completes the proof of the proposition.
\end{proof}

\section{The conditional variance of the Markov binomial distribution}\label{varcond}

Here we are interested in the variance of $K_n$ given the state of
the chain at time $n$. Let $K_n^\tau$ be the random variable $K_n$
conditioned on $Y_n=\tau\in\{\f,\ab\}$. For completeness, we will
first give the corresponding means $\expec{\kf}$ and $\expec{\kab}$
which were also given in \cite{ MR0109363, MR0282416, Viveros}.
Using (\ref{P^n}) and (\ref{P(Y_k)}) we obtain that
\begin{equation}\label{expecfr}
\begin{split}
  \expec{\kf}=&~\expec{K_n\,|\, Y_n=\f}=\sum_{k=1}^n\prc{Y_k=\f}{Y_n=\f}\\
  =&~\sum_{k=1}^n\frac{\prc{Y_n=\f}{ Y_k=\f}\prob{Y_k=\f}}{\prob{Y_n=\f}}=\frac{\sum_{k=1}^n P^{n-k}(\f,\f)\prob{Y_k=\f}}{\st_\f(1-\ef\ei^{n-1})}\\
  =&~\frac{\sum_{k=1}^n(\st_\f+\st_\ab\ei^{n-k})\st_\f(1-\ef\ei^{k-1})}{\st_\f(1-\ef\ei^{n-1})}\\
  =&~n\,\frac{\st_\f-\ef\st_\ab\,\ei^{n-1}}{1-\ef\ei^{n-1}}+\frac{(\st_\ab-\ef\st_\f)(1-\ei^n)}{(1-\ei)(1-\ef\ei^{n-1})},
\end{split}
\end{equation}
and similarly,
\begin{equation}\label{expecab}
\expec{\kab}=n\,\frac{\st_\f-\ea\st_\ab
\ei^{n-1}}{1-\ea\ei^{n-1}}
  +\frac{(\ea\st_\ab-\st_\f)(1-\ei^n)}{(1-\ei)(1-\ea\ei^{n-1})}
  .
\end{equation}

\begin{prop}\label{prop:condition_var_Kn}
The variances of $K_n^\tau$, a $Bin(n,a,b,\In)$ distributed random
variable $K_n$ conditioned on $Y_n=\tau\in\{\f,\ab\}$, are given by
\begin{eqnarray*}
  &&\var{\kf}=n^2\,\frac{\st_\f^2-\ef\,\st_\ab^2\ei^{n-1}}{1-\ef\ei^{n-1}}
  -\Big(n\,\frac{\st_\f-\ef\st_\ab\,\ei^{n-1}}{1-\ef\ei^{n-1}}
  +\frac{(\st_\ab-\ef\st_\f)(1-\ei^n)}{(1-\ei)(1-\ef\ei^{n-1})}\Big)^2\\
 &&\hspace{2cm}-n\Big(\frac{\st_\ab\,\st_\f(1+3\ef\ei^{n-1})}{1-\ef\ei^{n-1}}
  +2\,\frac{\ef\,\st_\f^2+\st_\ab^2\ei^n-2\st_\ab\,\st_\f(1+\ef\ei^{n-1})}{(1-\ei)(1-\ef\ei^{n-1})}\Big)\\
  &&\hspace{1cm}+(1-\ei^n)\Big(\frac{\st_\ab\st_\f(4+\ef)-(\st_\ab+\ef\st_\f^2)}{(1-\ei)(1-\ef\ei^{n-1})}
  +2\,\frac{\ef\,\st_\f^2+\st_\ab^2-2\st_\ab\,\st_\f(1+\ef)}{(1-\ei)^2(1-\ef\ei^{n-1})}\Big),
\end{eqnarray*}
  and
\begin{eqnarray*}
&&\var{\kab}=n^2\,\frac{\st_\f^2-\ea\,\st_\ab^2\ei^{n-1}}{1-\ea\ei^{n-1}}
  -\Big(n\,\frac{\st_\f-\ea\st_\ab\,\ei^{n-1}}{1-\ea\ei^{n-1}}
  +\frac{(\ea\st_\ab-\st_\f)(1-\ei^n)}{(1-\ei)(1-\ea\ei^{n-1})}\Big)^2\\
  &&\hspace{0.6cm}-n\Big(\frac{\st_\ab\,\st_\f(1+(2+\ea)\ei^{n-1})}{1-\ea\ei^{n-1}}
  +2\,\frac{\st_\f^2+\ea\,\st_\ab^2\ei^n-\st_\ab\,\st_\f(1+\ea)(1+\ei^{n-1})}{(1-\ei)(1-\ea\ei^{n-1})}\Big)\\
  &&\hspace{1cm}+(1-\ei^n)\Big(\frac{\st_\ab\st_\f(4+\ea)-(\st_\f+\ea\st_\ab^2)}{(1-\ei)(1-\ea\ei^{n-1})}
  +2\,\frac{\st_\f^2+\ea\,\st_\ab^2-2\st_\ab\,\st_\f(1+\ea)}{(1-\ei)^2(1-\ea\ei^{n-1})}\Big).
\end{eqnarray*}
\end{prop}
\begin{proof}
  Since the calculation of $\var{\kab}$ is similar to $\var{\kf}$, we only deal with $\var{\kf}$.
  Note that $\var{\kf}=\expec{(\kf)^2}-(\expec{\kf})^2$. Using
  (\ref{expecfr}) it suffices to calculate
  \begin{equation*}
  \begin{split}
    \expec{(\kf)^2}=&~\expec{\Big(\sum_{k=1}^n{\bf 1}_{\{Y_k=\f\}}\Big)^2\,\Big|\, Y_n=\f}\\
    =&~\sum_{k=1}^n\prc{Y_k=\f}{Y_n=\f}+2\sum_{1\le i<j\le n}\prc{Y_i=\f, Y_j=\f}{Y_n=\f}\\
    =&~\expecfr{K_n}+2\sum_{1\le i<j\le n}\prc{Y_i=\f,
    Y_j=\f}{Y_n=\f}.
    \end{split}
  \end{equation*}
It follows from (\ref{P^n}) and (\ref{P(Y_k)}) that
 \begin{equation*}
  \begin{split}
    &~\prc{Y_i=\f, Y_j=\f}{Y_n=\f}=\frac{\prob{Y_i=\f} \prc{Y_j=\f, Y_n=\f}{Y_i=\f}}{\prob{Y_n=\f}}\\
    =&~\frac{\prob{Y_i=\f} P^{j-i}(\f,\f) P^{n-j}(\f,\f)}{\prob{Y_n=\f}}
    =\frac{(1-\ef\ei^{i-1})(\st_\f+\st_\ab\ei^{j-i})(\st_\f+\st_\ab\ei^{n-j})}{1-\ef\ei^{n-1}}\\
    =&~\frac{\st_\ab^2\ei^{n-i}-\st_\f^2\ef\ei^{i-1}}{1-\ef\ei^{n-1}}
    +\frac{\st_\f^2-\st_\ab^2\ef\ei^{n-1}}{1-\ef\ei^{n-1}}+\frac{\st_\ab\st_\f(\ei^{n-j}-\ef\ei^{j-1})}{1-\ef\ei^{n-1}}\\
    &\hspace{5.2cm}+\,\frac{\st_\ab\st_\f(\ei^{j-i}-\ef\ei^{n-1-(j-i)})}{1-\ef\ei^{n-1}}.
  \end{split}
  \end{equation*}
  Performing the eight summations in the above equation we obtain
  that
 \begin{equation*}
  \begin{split}
 &~2\sum_{1\le i<j\le n}\prc{Y_i=\f,
    Y_j=\f}{Y_n=\f}\\
    =&~\Big(2(1-\ei^n)\frac{\st_\f^2\ef+\st_\ab^2\ei}{(1-\ef\ei^{n-1})(1-\ei)^2}
    -2n\,\frac{\st_\f^2\ef+\st_\ab^2\ei^n}{(1-\ef\ei^{n-1})(1-\ei)}\Big)    \\
    &~+\,\frac{n(n-1)(\st_\f^2-\st_\ab^2\ef\ei^{n-1})}{1-\ef\ei^{n-1}}+\frac{2\st_\f\st_\ab}{1-\ef\ei^{n-1}}\Big(n\frac{1+\ef\ei^n}{1-\ei}-(1-\ei^n)\frac{1+\ef\ei}{(1-\ei)^2}\Big)\\
    &\hspace{4cm}+\,\frac{2\st_\f\st_\ab}{1-\ef\ei^{n-1}}\Big(n\frac{\ei+\ef\ei^{n-1}}{1-\ei}-(1-\ei^n)\frac{\ei+\ef}{(1-\ei)^2}\Big)\\
    =&~n(n-1)\frac{\st_\f^2-\st_\ab^2\ef\ei^{n-1}}{1-\ef\ei^{n-1}}
    +2n\,\frac{\st_\f\st_\ab(1+\ei)(1+\ef\ei^{n-1})-(\st_\f^2\ef+\st_\ab^2\ei^n)}{(1-\ef\ei^{n-1})(1-\ei)}\\
    &\hspace{4.1cm}+2(1-\ei^n)\frac{\st_\f^2\ef+\st_\ab^2\ei-\st_\f\st_\ab(1+\ei)(1+\ef)}{(1-\ef\ei^{n-1})(1-\ei)^2},
 \end{split}
  \end{equation*}
which, combined with (\ref{expecfr}), yields the expression for
$\var{\kf}$.
\end{proof}

 For the special
  initial distributions $(0,1)$ and $(1,0)$, we have the excentricities
  $\ef\big((0,1)\big)=1=\ea\big((1,0)\big)$. Substituting them in equations (\ref{expecfr}), (\ref{expecab}) and Proposition
  \ref{prop:condition_var_Kn} we obtain that
  $$\expecab{\kf}=\expecfr{\kab},\quad \varab{\kf}=\varfr{\kab},$$
  where
  $$\rm{E}_{_{\scriptstyle\ab}}:=\rm{E}_{(0,1)},~~
  \rm{E}_{_{\scriptstyle\f}}:=\rm{E}_{(1,0)},~~
  \rm{Var}_{_{\scriptstyle\ab}}:=\rm{Var}_{(0,1)},~~\rm{Var}_{_{\scriptstyle\f}}:=\rm{Var}_{(1,0)}.$$
  More generally we have the following.
\begin{prop}
  For any ${\mathit Bin}(n,a,b,\In)$ distributed random variable $K_n$ and any positive integer $m$, the $m^{\rm th}$ moment of $\kf$
   conditioned on $Y_1=\ab$ is equal to the $m^{\rm th}$ moment of $\kab$ conditioned on $Y_1=\f$, i.e.,
   for $m=1,2,\dots$
   $$
   \expecab{(\kf)^m}=\expecfr{(\kab)^m}.
   $$

\end{prop}
\begin{proof}
  Note that for $m\le n$
 \begin{equation*}
  \begin{split}
    K_n^m=\Big(\sum_{k=1}^n{\bf 1}_{\{Y_k=\f\}}\Big)^m=&~C_1\sum_{k=1}^n{\bf 1}_{\{Y_k=\f\}}+C_2 \sum_{i_1<i_2}{\bf 1}_{\{Y_{i_1}=\f, Y_{i_2}=\f\}}\\
   &+\cdots+C_m\sum_{i_1<i_2<\dots<i_m}{\bf 1}_{\{Y_{i_1}=\f,Y_{i_2}=\f,\dots,Y_{i_m}=\f\}},
 \end{split}
  \end{equation*}
 where the $C_i$'s are constants related to $n$ and $m$.  This implies
 that for $\tau\in\{\f,\ab\}$
 \begin{equation*}
  \begin{split}
    \expectau{(K_n^\tau)^m}=&~C_1\sum_{k=1}^n \mathrm{P}_{\bar{\tau}}(Y_k=\f\,|\, Y_n=\tau)+C_2\sum_{i_1<i_2}
    \mathrm{P}_{\bar{\tau}}(Y_{i_1}=\f, Y_{i_2}=\f\,|\, Y_n=\tau)\\
    &+\cdots+C_m\sum_{i_1<i_2<\dots<i_m}\mathrm{P}_{\bar{\tau}}(Y_{i_1}=\f, Y_{i_2}=\f,\dots,Y_{i_m}=\f\,|\, Y_n=\tau),
 \end{split}
  \end{equation*}
  where $\bar{\f}=\ab, \bar{\ab}=\f$ and $\mathrm{P}_\ab:=\mathrm{P}_{(0,1)}, \mathrm{P}_\f:=\mathrm{P}_{(1,0)}$.

  Thus we only need to show that for $1\le i_1<\dots<i_k\le n$,
  \begin{equation}\label{eq:time reversibel}
\mathrm{P}_\ab(Y_{i_1}=\f,\dots,Y_{i_k}=\f\,|\,
Y_n=\f)=\mathrm{P}_\f(Y_{n-i_k+1}=\f,\dots,Y_{n-i_1+1}=\f\,|\,
Y_n=\ab).
  \end{equation}
 It is easy to see that both sides of Equation (\ref{eq:time reversibel}) equal $0$ if $i_1=1$. Now suppose $i_1\ge 2$.
 Since $\{Y_k,k\ge 1\}$ is a homogeneous time reversible Markov chain,
 we have
 \begin{equation*}
  \begin{split}
  &~\mathrm{P}_\ab(Y_{i_1}=\f,\dots,Y_{i_k}=\f\,|\, Y_n=\f)\\
  =&~\frac{\mathrm{P}_\ab(Y_n=\f\,|\,Y_{i_k}=\f)\mathrm{P}_\ab(Y_{i_k}=\f\,|\, Y_{i_{k-1}}=\f)\cdots\mathrm{P}_\ab(Y_{i_2}=\f\,|\, Y_{i_1}=\f)
  \mathrm{P}_\ab(Y_{i_1}=\f)}{\mathrm{P}_\ab(Y_n=\f)}\\
  =&~\frac{P^{n-i_k}(\f, \f) P^{i_k-i_{k-1}}(\f, \f) \cdots P^{i_2-i_1}(\f,\f) P^{i_1-1}(\ab, \f)}{P^{n-1}(\ab, \f)}\\
  =&~\frac{P^{n-i_k}(\f,\f) P^{i_k-i_{k-1}}(\f, \f) \cdots P^{i_2-i_1}(\f, \f) \frac{\st_\f}{\st_\ab}P^{i_1-1}(\f,\ab)}{\frac{\st_\f}{\st_\ab}P^{n-1}(\f,\ab)}\\
  =&~\mathrm{P}_\f(Y_{n-i_k+1}=\f,Y_{n-i_{k-1}+1}=\f,\dots,Y_{n-i_1+1}=\f\,|\, Y_n=\ab),
 \end{split}
  \end{equation*}
which yields Equation (\ref{eq:time reversibel}). Thus the
proposition is established for $m\le n$. In a similar way, one can
show that the proposition holds for all $m> n$.
\end{proof}

\section{The probability mass function of the Markov binomial distribution}\label{PMF}
For any $Bin(n,a,b,\nu)$ distributed random variable $K_n$, we will
give sufficient and necessary conditions for the probability mass
function of $K_n$ to be unimodal, bimodal or trimodal. These three
kinds of shapes are mentioned by Viveros et al.~\cite{Viveros}
without any further explanation.

 Given $n\ge 1$, let $f_{n}$ be the probability mass function of $K_n$, i.e.,
 $$f_{n}(j)=\prob{K_n=j}. $$
Particularly, $f_n(j)=0$ if $j<0$ or $j>n$. By an easy computation,
\begin{equation*}
\begin{split}
  f_{n+2}(j+1)=&~\prob{K_{n+1}=j+1,Y_{n+1}=\ab}P(\ab,\ab)+\prob{K_{n+1}=j,Y_{n+1}=\ab}P(\ab,\f)\\
  &+\prob{K_{n+1}=j+1,Y_{n+1}=\f}P(\f,\ab)+\prob{K_{n+1}=j, Y_{n+1}=\f}P(\f,\f)\\
  =&~f_{n+1}(j+1)P(\ab,\ab)+\prob{K_{n+1}=j+1,Y_{n+1}=\f}\big(P(\f,\ab)-P(\ab,\ab)\big)\\
  &\qquad\quad+f_{n+1}(j)P(\f,\f)+\prob{K_{n+1}=j,Y_{n+1}=\ab}\big(P(\ab,\f)-P(\f,\f)\big)\\
  =&~P(\ab,\ab)f_{n+1}(j+1)+P(\f,\f)f_{n+1}(j)+\big(P(\f,\ab)-P(\ab,\ab)\big)f_{n}(j),
\end{split}
\end{equation*}
where the last equality holds since
$$
P(\f,\ab)+P(\f,\f)=P(\ab,\ab)+P(\ab,\f)=1.
$$
Substituting (\ref{transition matrix}) in the above recursion equation yields that for $n\ge 1$
\begin{equation}\label{eq:recursion_prob_mass_func}
f_{n+2}(j+1)=(1-b)f_{{n+1}}(j+1)+(1-a)f_{{n+1}}(j)-(1-a-b)f_{{n}}(j)
\end{equation}
with initial conditions
\begin{equation}\label{eq:initial conditions}
\begin{split}
&f_1(0)=\In_\ab,\quad f_{1}(1)=\In_\f;\\
&f_2(0)=\In_\ab (1-b),\quad f_2(1)=\In_\ab b+\In_\f a,\quad f_2(2)=\In_\f(1-a).
\end{split}
\end{equation}
In \cite{MR0109363, MR0282416, Viveros} (implicit) expressions for
the probability mass function of $K_n$ are given, but the closed
form presented here is more helpful to study its shape.

\begin{prop}\label{prop:fn}
The probability mass function $f_n$ of a  $Bin(n,a,b,\In)$ distributed random variable $K_n$ can be written as
$$
f_n(j)=\left\{
\begin{array}{ll}
  \In_\ab(1-b)^{n-1}&\quad j=0,\\
  (1-b)^{n-j}(1-a)^{j-1}\sum\limits_{k=0}^{j-1}\binom{j-1}{k}\delta^k c_{j-1,k}(n)&\quad 1\le j\le n-1,\\
  \In_\f (1-a)^{n-1}&\quad j=n,\\
  0&\quad \textrm{otherwise},
\end{array}
\right.
$$
where $\delta=a b/\big((1-a)(1-b)\big)$ and
$$
c_{j,k}(n)=\In_\f\binom{n-2-j}{k-1}+\frac{\In_\f a+\In_\ab b }{1-b}\binom{n-2-j}{k}+\frac{\In_\ab a b}{(1-b)^2}\binom{n-2-j}{k+1}.
$$
\end{prop}
\begin{proof}
  It is easy to see that the recursion equation (\ref{eq:recursion_prob_mass_func})
  with initial conditions (\ref{eq:initial conditions}) has a unique solution. We only need to check
   that $f_n$ presented in the proposition satisfies the equations (\ref{eq:recursion_prob_mass_func}) and  (\ref{eq:initial conditions}),
    and that the summation of $f_n(j)$ from $j=0$ to $n$ equals 1. It is easy to see that (\ref{eq:recursion_prob_mass_func}) holds for $j< 0$ and $j> n$. Equation (\ref{eq:recursion_prob_mass_func}) holds for $j=0$ since
  \begin{equation*}
  \begin{split}
    &~(1-b)f_{n+1}(1)+(1-a)f_{n+1}(0)-(1-a-b)f_n(0)\\
    =&~(1-b)^{n+1}c_{0,0}(n+1)+\In_\ab(1-a)(1-b)^n-\In_\ab(1-a-b)(1-b)^{n-1}\\
    =&~(1-b)^{n-1}\big((1-b)(\In_\f a+\In_\ab b)+(n-1)\In_\ab a b+\In_\ab (1-a)(1-b)-\In_\ab(1-a-b)\big)\\
    =&~(1-b)^{n-1}\big((1-b)(\In_\f a+\In_\ab b)+n\In_\ab a b\big)=\,f_{n+2}(1).
    \end{split}
  \end{equation*}
  Similarly, Equation (\ref{eq:recursion_prob_mass_func}) holds for $j=n$.

   Suppose now $1\le j\le n-1$. From simple properties of the binomial coefficients in $c_{j,k}(n)$ it follows that
$$
c_{j-1,k}(n)=c_{j,k}(n+1)= c_{j+1,k}(n+2),
$$ and
\begin{equation}\label{eq:recursion_c_jk}
c_{j,k}(n+2)=c_{j+1,k}(n+2)+c_{j+1,k-1}(n+2).
\end{equation}
We write $c_{j,k}:=c_{j,k}(n+2)$ for short. Thus
\begin{equation*}
\begin{split}
  &~(1-b)f_{n+1}(j+1)+(1-a)f_{n+1}(j)-(1-a-b)f_{n}(j)\\
  =&~(1-b)^{n+1-j}(1-a)^{j}\sum_{k=0}^{j}\binom{j}{k}\delta^k c_{j+1,k}+(1-b)^{n+1-j}(1-a)^{j}\sum_{k=0}^{j}\binom{j-1}{k}\delta^k c_{j,k}\\
   &\hspace{5.3cm}-(1-\delta)(1-b)^{n+1-j}(1-a)^{j}\sum_{k=0}^{j}\binom{j-1}{k}\delta^k c_{j+1,k}\\
  =&~(1-b)^{n+1-j}(1-a)^{j}\left[\sum_{k=0}^{j}\binom{j-1}{k}\delta^{k}c_{j+1,k}
  +\sum_{k=0}^{j}\binom{j-1}{k-1}\delta^{k}c_{j+1,k}\right.\\
   &\hspace{2cm}\left.+\sum_{k=0}^{j}\binom{j-1}{k}\delta^k c_{j,k}-\sum_{k=0}^{j}\binom{j-1}{k}\delta^k c_{j+1,k}+\sum_{k=0}^{j}\binom{j-1}{k-1}\delta^k c_{j+1,k-1}\right]\\
  =&~(1-b)^{n+1-j}(1-a)^{j}\left[\sum_{k=0}^{j}\binom{j-1}{k-1}\delta^{k}(c_{j+1,k}+c_{j+1,k-1})+\sum_{k=0}^{j}\binom{j-1}{k}\delta^k c_{j,k}\right]\\
  =&~(1-b)^{n+1-j}(1-a)^{j}\left[\sum_{k=0}^{j}\binom{j-1}{k-1}\delta^{k}c_{j,k}+\sum_{k=0}^{j}\binom{j-1}{k}\delta^k c_{j,k}\right]\\
  =&~(1-b)^{n+1-j}(1-a)^{j}\sum_{k=0}^{j}\binom{j}{k}\delta^{k}c_{j,k}=f_{n+2}(j+1).
  \end{split}
\end{equation*}
Now we are going to show by induction that  $\sum_{j=0}^n f_n(j)=1$
for each $n\ge 1$. For $n=1$ and $2$, we have $
f_1(0)+f_1(1)=\In_\ab+\In_\f=1,$ and
$$
f_2(0)+f_2(1)+f_2(2)=\In_\ab(1-b)+\In_\ab b+\In_\f a+\In_\f(1-a)=1.
$$
Suppose $f_n$ and $f_{n+1}$ are probability mass functions, then by Equation (\ref{eq:recursion_prob_mass_func}) \begin{equation*}
\begin{split}
 \sum_{j=0}^{n+2}
 f_{n+2}(j)&=(1-a)\sum_{j=0}^{n+2}f_{n+1}(j)+(1-b)\sum_{j=0}^{n+2}f_{n+1}(j-1)-(1-a-b)\sum_{j=0}^{n+2}f_{n}(j-1)\\
 &=(1-a)+(1-b)-(1-a-b)=1.
 \end{split}
\end{equation*}
This completes the proof.
\end{proof}

\begin{ex}
 Let $n=200, a=0.01, b=0.03$ and $\In=(0.1,
0.9)$. By Proposition \ref{prop:fn} we obtain the probability mass
function of $K_{200}$ shown in Figure \ref{pic0}. Apparently
$f_{200}$ is trimodal.
\begin{figure}[h]
 \centering{ \includegraphics[width=8cm]{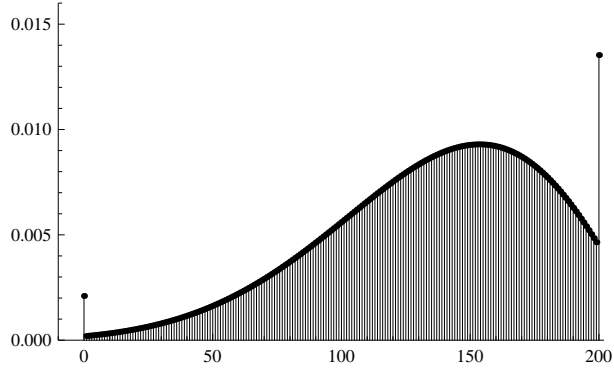}\\
  \caption{Probability mass function  $f_{200}$ of $K_{200}$ with $a=0.01, b=0.03$ and $\In=(0.1, 0.9)$.}\label{pic0}}
\end{figure}
\end{ex}

A finite sequence of real numbers $\{x_i\}_{i=0}^n$ is said to be
\emph{unimodal} if there exists an index $0\le n^*\le n$, called a
\emph{mode} of the sequence, such that $x_0\le x_1\le\dots\le
x_{n^*}$ and $x_{n^*}\ge x_{n^*+1}\ge\dots\ge x_n$. In particular,
we call the sequence $\{x_i\}_{i=0}^n$ \emph{strictly unimodal} if
all modes $n^*$ satisfy $0<n^*<n$. From the definition it is easy to
see that a monotonic sequence is unimodal.

A nonnegative sequence $\{x_i\}_{i=0}^n$ is called
\emph{log-concave} (or \emph{strictly log-concave}) if $x_{i-1}x_{i+1}\le x_i^2$ (or $x_{i-1}x_{i+1}< x_i^2$) for all $1\le i\le
n-1$. It is well known that the sequence $\{x_i\}_{i=0}^n$ is
log-concave if and only if $x_{i_1-1}x_{i_2+1}\le x_{i_1} x_{i_2}$
for all $1\le i_1\le i_2\le n-1$. Moreover, log-concavity implies unimodality.

The definitions of unimodality and log-concavity can be extended
naturally  to infinite sequences.

\begin{prop}
  Let $a+b\ge 1$, and let $f_n$ be the probability mass function of a $Bin(n,a ,b,\In)$ distributed random variable $K_n$. Then the sequence $\{f_n(j)\}_{j=0}^n$ is log-concave, and hence unimodal. Moreover, the mode $n^*$ satisfies $\lfloor\expec{K_n}\rfloor\le n^*\le \lceil\expec{K_n}\rceil$.
\end{prop}
\begin{proof}
 Let $G_n$ be the generating function of $K_n$, i.e., for all real $s$
  $$
  G_n(s)=\expec{s^{K_n}}=\sum_{j=0}^n f_n(j)s^j.
  $$
 Without loss of generality we suppose $0<\In_\f<1$. Then (by Proposition \ref{prop:fn}) $ G_n$ has positive coefficients.  It follows from the recursion equation (\ref{eq:recursion_prob_mass_func}) that
  $$
  G_{n+2}(s)=\big((1-a)s+(1-b)\big)G_{n+1}(s)-(1-a-b)s G_{n}(s).
  $$
Since $a+b\ge 1$, we obtain by Corollary 2.4 of \cite{ZEROS} that
for each $n\ge 1$ all zeros of $G_n$ are real. Thus the sequence
$\{f_n(j)\}_{j=0}^n$ is log-concave and hence unimodal with mode
$n^*$ between $\lfloor\expec{K_n}\rfloor$ and
$\lceil\expec{K_n}\rceil$.
\end{proof}

When $a+b<1$, Figure \ref{pic0} suggests that the probability mass
function $\{f_n(j)\}_{j=0}^n$ is not unimodal. However, Figure
\ref{pic0} also suggests that $\{f_n(j)\}_{j=1}^{n-1}$ is unimodal.
We will indeed show in Proposition \ref{prop:unimodal} that the
sequence $\{f_n(j)\}_{j=1}^{n-1}$ is log-concave, implying
unimodality. In order to prove Proposition \ref{prop:unimodal} it is
helpful to use the following lemma which can be derived directly
 from Lemma 2.2 and Proposition 2.4 of \cite{LOGCONCAVITY}. To be more self-contained, we give a proof by using  simple properties of binomial coefficients and log-concave sequences.

\begin{lem}\label{lemm:LC-positive}
  For any positive integer $j$ and a nonnegative log-concave sequence $\{x_k\}_k$, let $d_{j,k}:=\binom{j}{k}x_k$. Then for any $0\le 2\ell\le m\le 2j$,
  $$
\sum_{k=\ell}^{\lfloor m/2\rfloor}D_{j,k}(m)\ge 0,
  $$
  where for $k<m/2$
  $$
  D_{j,k}(m)=2 \,d_{j,k}\,d_{j,m-k}-d_{j-1,k}\,d_{j+1,m-k}-d_{j+1,k}\,d_{j-1,m-k},
  $$
and for $m$ even and $k=m/2$
  $$
  D_{j,k}(m)=d_{j,k}^2-d_{j-1,k}\,d_{j+1,k} .
  $$

\end{lem}

\begin{proof}
  Note that for $k<m/2$
  \begin{equation*}
    \begin{split}
      D_{j,k}(m)=&~\Big[2\binom{j}{k}\binom{j}{m-k}-\binom{j-1}{k}\binom{j+1}{m-k}-\binom{j+1}{k}\binom{j-1}{m-k}\Big]x_k x_{m-k}\\
      =&~\Big[\binom{j-1}{k-1}\binom{j}{m-k}-\binom{j}{k-1}\binom{j-1}{m-k}\Big]x_k x_{m-k}\\
      &\hspace{1cm}-\Big[\binom{j-1}{k}\binom{j}{m-k-1}-\binom{j}{k}\binom{j-1}{m-k-1}\Big]x_k x_{m-k}.
    \end{split}
  \end{equation*}
  For brevity, we only show the lemma for $m$ odd. Let $m=2s+1$. Then for $0\le \ell\le s<j$,
  \begin{equation*}
    \begin{split}
      \sum_{k=\ell}^s D_{j,k}(m)=&\sum_{k=\ell-1}^{s-1}\Big[\binom{j-1}{k}\binom{j}{m-k-1}-\binom{j}{k}\binom{j-1}{m-k-1}\Big]x_{k+1}x_{m-k-1}\\
      &-\sum_{k=\ell}^s\Big[\binom{j-1}{k}\binom{j}{m-k-1}-\binom{j}{k}\binom{j-1}{m-k-1}\Big]x_k x_{m-k}\\
      =&~\sum_{k=\ell}^{s-1}\Big[\binom{j-1}{k}\binom{j}{m-k-1}-\binom{j}{k}\binom{j-1}{m-k-1}\Big](x_{k+1}x_{m-k-1}-x_k x_{m-k})\\
      &+\Big[\binom{j-1}{\ell-1}\binom{j}{m-\ell}-\binom{j}{\ell-1}\binom{j-1}{m-\ell}\Big]x_\ell x_{m-\ell}\;      \ge\; 0,
    \end{split}
  \end{equation*}
  where the last inequality holds since $\binom{j-1}{k}\binom{j}{m-k-1}\ge\binom{j}{k}\binom{j-1}{m-k-1}$ for $k\le s-1$ and the sequence $\{x_k\}_k$ is log-concave.
This completes the proof of the lemma.
\end{proof}

Inspired by the proof of Theorem 3.10 of \cite{LOGCONCAVITY}, we are going to use Lemma \ref{lemm:LC-positive} to show the log-concavity of an important class of sequences.

\begin{lem}\label{lemma:log-concave}
  Let $\delta>0$ and $\{c_{j,k}\}_{j,k\in\mathbb{Z}}$ be an nonnegative double sequence satisfying
  $$
  c_{j,k}=c_{j+1,k}+c_{j+1,k-1},
  $$
  and $c_{j,k}=0$ for all $j\in\mathbb{Z}$ and $k\le -2$.
  Then the sequence
  $$
  \Big\{\sum_{k=0}^j\binom{j}{k}\delta^k c_{j,k}\Big\}_{j\ge 0}
  $$
  is log-concave.
\end{lem}
\begin{proof}
 We fix $j\ge 1$. Let $d_{j,k}:=\binom{j}{k}\delta^k$. We have to show that $z_j^2\ge z_{j-1}z_{j+1}$ where
  $$
z_j:=\sum_{k=0}^{j}\binom{j}{k}\delta^k c_{j,k}=\sum_{k=0}^j d_{j,k}\,c_{j,k}.
  $$
We use the short notation $v_k:=c_{j+1,k}$. Since
$c_{j,k}=c_{j+1,k}+c_{j+1,k-1}$, this yields
\begin{equation*}
 z_{j+1}=\sum_{k=0}^{j+1}d_{j+1,k}v_k,\quad z_j=\sum_{k=0}^{j}d_{j,k}(v_k+v_{k-1}), \quad
  z_{j-1}=\sum_{k=0}^{j-1}d_{j-1,k}(v_k+2\, v_{k-1}+v_{k-2}).
\end{equation*}
Note that $v_k=c_{j+1,k}=0$ for all $j$ and $k\le -2$ and
$d_{j,k}=0$ for $k<0$ or $k>j$, by the definition of $\binom{j}{k}$.
Rewrite
\begin{equation*}\begin{split}
  z_{j+1}&=\sum_{k=0}^{j+2} d_{j+1,k-1}v_{k-1},\quad z_{j}=\sum_{k=0}^{j+2}\big(d_{j,k-1}+d_{j,k}\big)v_{k-1}, \\ z_{j-1}&=\sum_{k=0}^{j+2}\big(d_{j-1,k-1}+2\,d_{j-1,k}+d_{j-1,k+1}\big)v_{k-1}.
\end{split}\end{equation*}
  Then $z_j^2-z_{j-1}z_{j+1}$ can be rewritten in a quadratic form of $j+3$ variables $v_{-1}, v_0, v_1,\dots,
  v_{j+1}$:
  $$
  z_j^2-z_{j-1}z_{j+1}=\sum_{m=0}^{2 (j+2)}\sum_{k=0}^{\lfloor m/2\rfloor}e_{j,k}(m)v_{k-1} v_{m-k-1},
  $$
  where
  \begin{equation*}
  \begin{split}
  e_{j,k}(m)=&~2\,\big(d_{j,k-1}+d_{j,k}\big)\big(d_{j,m-k-1}+d_{j,m-k}\big)\\
  &~-\big(d_{j-1,k-1}+2 \,d_{j-1,k}+d_{j-1,k+1}\big)d_{j+1,m-k-1}\\
  &~-d_{j+1,k-1}\big(d_{j-1,m-k-1}+2\, d_{j-1,m-k}+d_{j-1,m-k+1}\big).
  \end{split}
 \end{equation*}
 Since the $v_k$'s are all nonnegative, it suffices to show that $
 \sum_{k=0}^{\lfloor m/2\rfloor} e_{j,k}(m)\ge 0, $ for all $0\le m\le 2(j+2)$.
 Rewrite
 $$
 e_{j,k}(m)=P_k+2 \,Q_k+R_k,
 $$
 where
 \begin{equation*}\begin{split}
   P_k&=2\, d_{j,k-1}\,d_{j,m-k-1}-d_{j-1,k-1}\,d_{j+1,m-k-1}-d_{j+1,k-1}\,d_{j-1,m-k-1},\\
   Q_k&=d_{j,k-1}\,d_{j,m-k}+d_{j,k}\,d_{j,m-k-1}-d_{j-1,k}\,d_{j+1,m-k-1}-d_{j+1,k-1}\,d_{j-1,m-k},\\
   R_k&=2\, d_{j,k}\,d_{j,m-k}-d_{j-1,k+1}\,d_{j+1,m-k-1}-d_{j+1,k-1}\,d_{j-1,m-k+1}.
 \end{split}\end{equation*}
 Then we only need to show that
 \begin{equation*}
 \sum_{k=0}^{\lfloor m/2\rfloor} P_k\ge 0,\quad\sum_{k=0}^{\lfloor m/2\rfloor} Q_k\ge 0,\quad\sum_{k=0}^{\lfloor m/2\rfloor} R_k\ge 0.
 \end{equation*}
For brevity, we show this only for the case $m$ is odd. For $m$ even
the proof is very similar, but somewhat longer. Let $m=2 s+1$. It
follows from  Lemma \ref{lemm:LC-positive} that
$$
\sum_{k=0}^s P_k=\sum_{k=0}^s
D_{j,k-1}(m-2)=\sum_{k=0}^{s-1}D_{j,k}(m-2)\ge 0,
$$
where the second equality holds since $D_{j,k}(m-2)=0$ for $k<0$.
 Recalling from Lemma \ref{lemm:LC-positive} that $D_{j,s}(m-1)=d^2_{j,s}-d_{j-1,s}\, d_{j+1,s}$ we also have
\begin{equation*}
\begin{split}
  \sum_{k=0}^s Q_k&=\sum_{k=-1}^{s-1}(d_{j,k}\,d_{j,m-k-1}-d_{j+1,k}\,d_{j-1,m-k-1})+\sum_{k=0}^s (d_{j,k}\,d_{j,m-k-1}-d_{j-1,k}\,d_{j+1,m-k-1}) \\
  &=\sum_{k=0}^s D_{j,k}(m-1)\ge 0.
  \end{split}
\end{equation*}
Moreover,
 \begin{equation*}\begin{split}
   \sum_{k=0}^{s} R_k&=2\,\sum_{k=0}^{s} d_{j,k}\,d_{j,m-k}-\sum_{k=1}^{s+1}d_{j-1,k}\,d_{j+1,m-k}-\sum_{k=-1}^{s-1}d_{j+1,k}\,d_{j-1,m-k}\\
   &=\sum_{k=0}^{s}D_{j,k}(m)+d_{j-1,0}\,d_{j+1,m}\ge\sum_{k=0}^{s}D_{j,k}(m)\ge 0.
 \end{split}\end{equation*}
 This finishes the proof of the lemma.
 \end{proof}

\begin{prop}\label{prop:unimodal}
  For any $Bin(n,a,b,\In)$ distributed random variable $K_n$, let $f_n$ be its probability mass function. Then the sequence $\{f_{n}(j)\}_{j=1}^{n-1}$ is
  log-concave.
\end{prop}
\begin{proof}
According to Proposition \ref{prop:fn} we have that for $1\le j\le
n-1$
\begin{equation*}
  f_n(j)=(1-b)^{n-j}(1-a)^{j-1}\sum_{k=0}^{j-1}\binom{j-1}{k}\delta^k c_{j-1,k},
\end{equation*}
where $\delta>0,$ and the double sequence
$\{c_{j,k}\}_{j,k\in\mathbb{Z}}$ satisfies the recursion equation
\begin{equation*}
c_{j,k}=c_{j+1,k}+c_{j+1,k-1}
\end{equation*}
(cf.~Equation (\ref{eq:recursion_c_jk})), and $c_{j,k}=0$ for $k\le
-2$. It follows from Lemma \ref{lemma:log-concave} that the sequence
$\{f_n(j)\}_{j=1}^{n-1}$ is log-concave.
\end{proof}
In fact we can show by sharping the proof of Lemma
\ref{lemma:log-concave} that $\{f_n(j)\}_{j=1}^{n-1}$ is strictly
log-concave, i.e., $f_{n}(j)^2>f_n(j-1)f_n(j+1)$ for
$j=2,\dots,n-2$. Proposition \ref{prop:unimodal} implies that the
shape of the probability mass function of $K_n$, which can be
unimodal, bimodal or trimodal, is determined by the following six
values:
$$f_n(0),~~ f_n(1),~~ f_n(2), ~~f_n(n-2), ~~f_n(n-1)~~\textrm{and}~~ f_n(n).$$
\begin{thm}\label{th:shape of prob mass func}
  For any $Bin(n,a,b,\In)$ distributed random variable $K_n$, let $f_n$ be its probability mass function.  Then
$f_n$ is unimodal, except that

 $f_n$ is bimodal with one peak on the left if and only if $f_n(0)> f_n(1)\le f_n(2)$
 and either $f_n(n-1)\ge f_n(n)$ or $f_n(n-2)<f_n(n-1)<f_n(n)$;

  $f_n$ is bimodal with one peak on the right if and only if  $f_n(n-2)\ge f_n(n-1)<f_n(n)$ and either $f_n(0)\le f_n(1)$or
  $f_n(0)>f_n(1)>f_n(2)$;

$f_n$ is trimodal if and only if $f_n(0)>f_n(1)\le f_n(2)$ and
$f_n(n-2)\ge f_n(n-1)<f_n(n)$.
\end{thm}

\begin{ex}
 We consider the special case
$\In=\st=(b/(a+b),a/(a+b))$ and $n=50$.
\begin{figure}[h]
  \centering{\includegraphics[width=5cm]{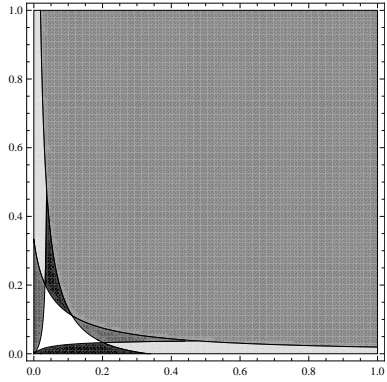}\\
  \caption{(I) When $(a,b)$ is in the gray region, $f_{50}$ is strictly unimodal;
  when $(a,b)$ is in the lower light gray region, $f_{50}$ is decreasing;
  when $(a,b)$ is in the upper light gray region, $f_{50}$ is increasing;
  (II) When $(a,b)$ is in the dark gray region, $f_{50}$ is bimodal with one peak on the left;
   when $(a,b)$ is in the black region, $f_{50}$ is bimodal with one peak on the right;
  (III) When $(a,b)$ is in the white region, $f_{50}$ is trimodal.}\label{pic00}}
\end{figure}
It follows from Proposition \ref{prop:fn}
that
$$
 f_n(0)=(1-b)^{n-1}\frac{a}{a+b},\quad  f_{n}(1)=\frac{(1-b)^{n-2} a b}{a+b}\Big(2+(n-2)\frac{a}{1-b}\Big),
$$
and
$$
f_n(2)=\frac{(1-b)^{n-3} a b}{a+b}\Big\{(1-a)\Big(2+(n-3)\frac{a}{1-b}\Big)+b\Big(1+(n-3)\frac{2 a}{1-b}+\big(\frac{a}{1-b}\big)^2\binom{n-3}{2}\Big)\Big\}.
$$
In a similar way one obtains the formulas for $f_n(n-2),f_n(n-1)$ and $f_n(n)$.
Figure \ref{pic00} is obtained via Theorem \ref{th:shape of prob mass func}. For some examples of probability mass functions in this class, see Figure \ref{pic01}.
\begin{figure}[h]
 \centering{ \includegraphics[width=5cm]{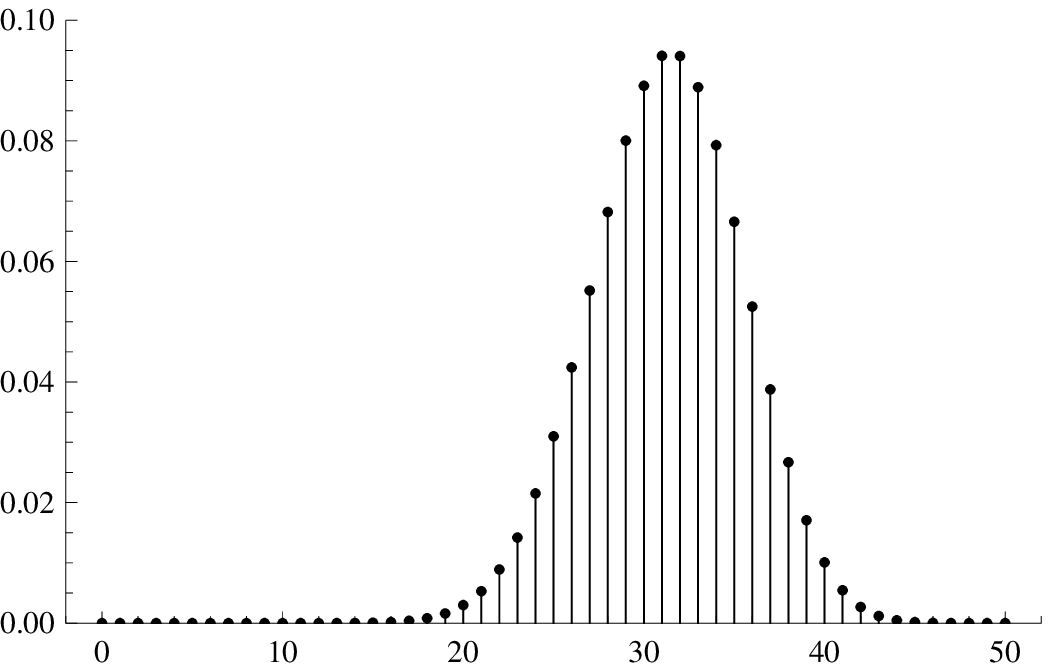}\quad\quad\includegraphics[width=5cm]{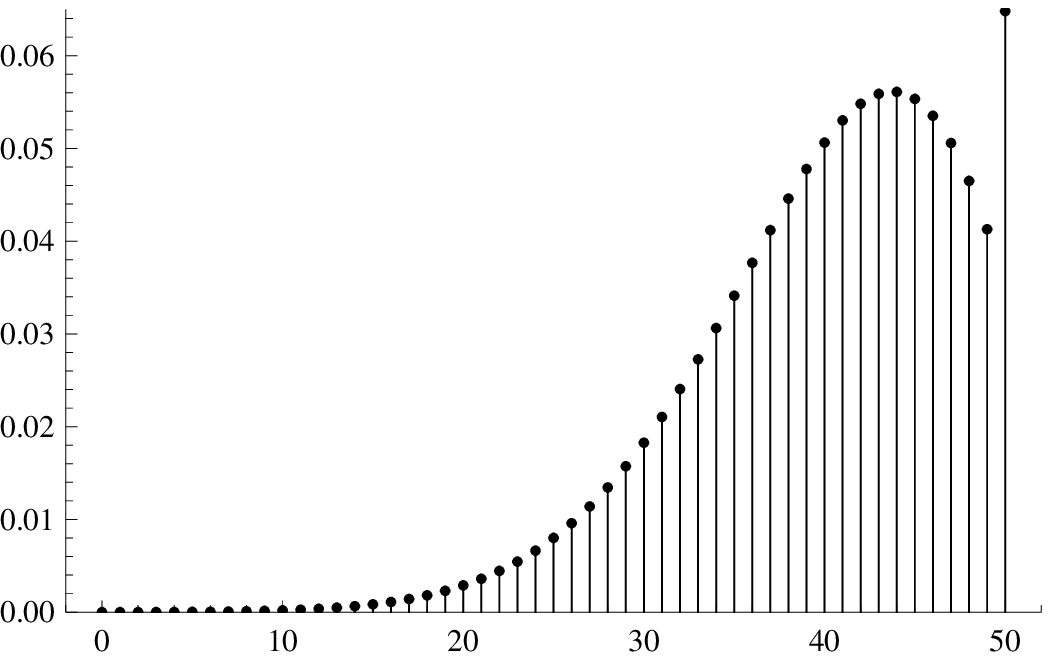}\\

 \medskip

  \includegraphics[width=5cm]{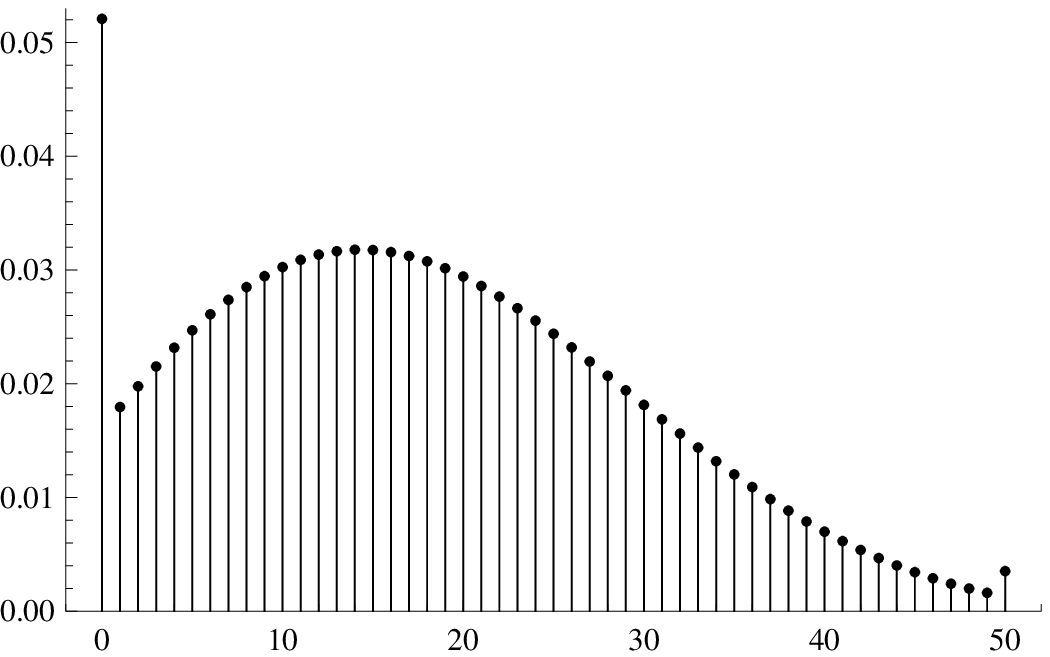}\quad\quad\includegraphics[width=5cm]{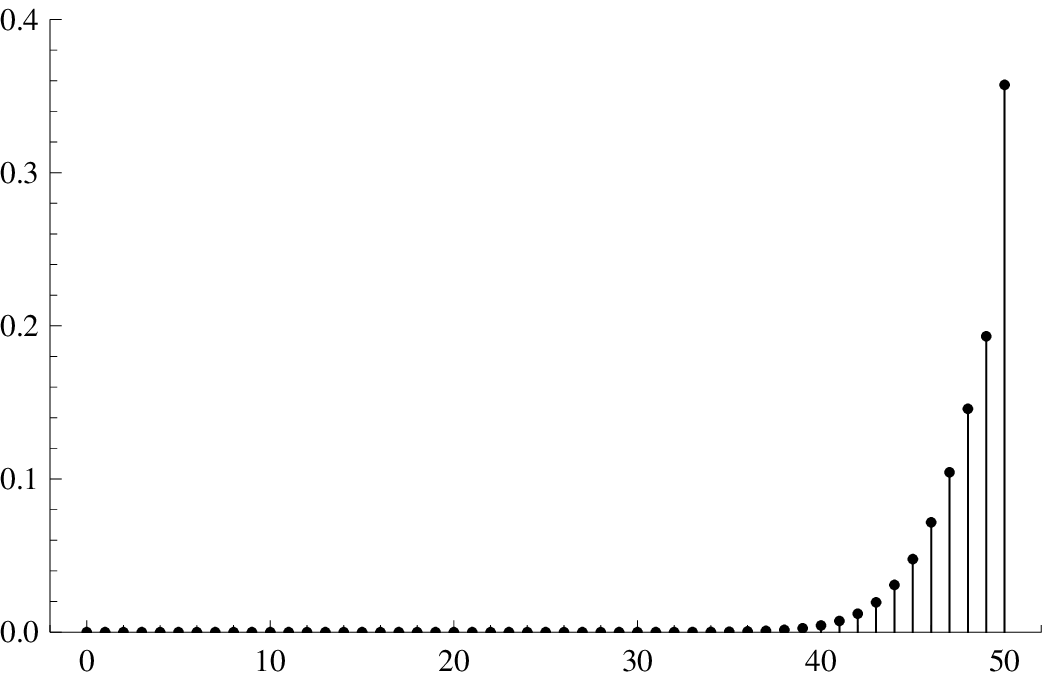}\\
  \caption{Probability mass function of $K_{50}$ with $ \In=\st$. In the upper left graph $a = 0.3, b=0.5$ and in the upper right graph $a=0.05,b=0.2$. In the lower left graph $a = 0.09, b=0.05$, and in the lower right  graph $a=0.02, b=0.5$.}\label{pic01}}
\end{figure}
\end{ex}

\section{The conditional probability mass functions}\label{sec:condmass}
For any $Bin(n,a,b,\In)$ distributed random variable $K_n$, let
$f^{\tau}_n$ be the probability mass function of $K_n^\tau$ with
$\tau\in\{\f,\ab\}$ , i.e.,
$$
f_n^{\tau}(j)=\prob{K_n^\tau=j}=\prob{K_n=j\,|\,  Y_n=\tau}.
$$
In order to deal with $f^{\tau}_n$ it is simpler to deal with the
partial probability mass functions
$$
\hat{f}_n^{\tau}(j)=\prob{K_n=j,~
Y_n=\tau}=f^{\tau}_n(j)\prob{Y_n=\tau}.
$$
 Since $\jab_n=f_n-\jf_n$, we only deal with $\jf_n$. It is easy to
obtain the recursion equation:
\begin{equation*}
\jf_{n+2}(j+1)=(1-b)\jf_{n+1}(j+1)+(1-a)\jf_{n+1}(j)-(1-a-b)\jf_{n}(j),
\end{equation*}
with initial conditions
\begin{equation*}
  \begin{split}
    \jf_1(0)&=0,\quad \jf_1(1)=\In_\f;\\
    \jf_2(0)&=0, \quad\jf_2(1)=\In_\ab b, \quad\jf_2(2)=\In_\f(1-a).
  \end{split}
\end{equation*}
Then  we obtain
the following proposition in a similar way as Proposition \ref{prop:fn}.
\begin{prop}\label{prop:fn_F}
  The partial probability mass function $\jf_n$ of a $Bin(n,a,b,\In)$ distributed random variable $K_n$ can be written as
  $$
\jf_n(j)=\left\{
\begin{array}{ll}
  (1-b)^{n-j}(1-a)^{j-1}\sum\limits_{k=0}^{j-1}\binom{j-1}{k}\delta^k c^{\f}_{j-1,k}(n)&\quad 1\le j\le n-1,\\
  \In_\f (1-a)^{n-1}&\quad j=n,\\
  0&\quad \textrm{otherwise},
\end{array}
\right.
$$
where $\delta=a b/\big((1-a)(1-b)\big)$ and
$$
c^{\f}_{j,k}(n)=\In_\f\binom{n-2-j}{k-1}+\frac{\In_\ab b }{1-b}\binom{n-2-j}{k}.
$$
\end{prop}

From Lemma \ref{lemma:log-concave} it follows that the sequence
$\{\jf_n(j)\}_{j=0}^{n-1}$ is log-concave, and hence
$\{\ff_n(j)\}_{j=0}^{n-1}$ is log-concave. Thus, in contrast to $f_n$, $\ff_n$ can
not have a trimodal shape. The unimodal
or bimodal (with one peak on the right) shape of $\ff_n$ depends on the values of $\ff_n(j)$ for $ j=n-2,n-1,n$.

 Similarly, the shape of $\fab_n$ can only be unimodal or bimodal (with one peak on the left) depending on the values of $\fab_n(j)$ for $ j=0,1,2$.

\section*{Acknowledgement}
The second author is partially supported  by the National
Natural Science Foundation of China 10971069 and Shanghai Education Committee Project 11ZZ41.

\bibliography{Markov-bin-AKAD}

\begin{thebibliography}{10}

\bibitem{MR2327530}
V.~{\v{C}}ekanavi{\v{c}}ius and B.~Roos.
\newblock Binomial approximation to the {M}arkov binomial distribution.
\newblock {\em Acta Appl. Math.}, 96(1-3):137--146, 2007.

\bibitem{MR2485024}
Vydas {\v{C}}ekanavi{\v{c}}ius and Bero Roos.
\newblock Poisson type approximations for the {M}arkov binomial distribution.
\newblock {\em Stochastic Process. Appl.}, 119(1):190--207, 2009.

\bibitem{Daugman}
John Daugman.
\newblock The importance of being random: statistical principles of iris
  recognition.
\newblock {\em Pattern Recognition}, 36(2):279--291, 2003.

\bibitem{iris}
Michel Dekking and Andr{\'e} Hensbergen.
\newblock A problem with the assessment of an iris identification system.
\newblock {\em SIAM Rev.}, 51(2):417--422, 2009.

\bibitem{SKTM}
Michel Dekking and DeRong Kong.
\newblock A simple stochastic kinetic transport model.
\newblock {\em In preparation}, 2011.

\bibitem{MR0058150}
R.~L. Dobru{\v{s}}in.
\newblock Limit theorems for a {M}arkov chain of two states.
\newblock {\em Izvestiya Akad. Nauk SSSR. Ser. Mat.}, 17:291--330, 1953.

\bibitem{Edwards}
A.W.F. Edwards.
\newblock The meaning of binomial distribution.
\newblock {\em Nature, London}, 186:1074, 1960.

\bibitem{MR0109363}
K.~R. Gabriel.
\newblock The distribution of the number of successes in a sequence of
  dependent trials.
\newblock {\em Biometrika}, 46:454--460, 1959.

\bibitem{MR0282416}
H.~J. Helgert.
\newblock On sums of random variables defined on a two-state {M}arkov chain.
\newblock {\em J. Appl. Probability}, 7:761--765, 1970.

\bibitem{ZEROS}
Lily~L. Liu and Yi~Wang.
\newblock A unified approach to polynomial sequences with only real zeros.
\newblock {\em Adv. in Appl. Math.}, 38(4):542--560, 2007.

\bibitem{Markov}
A.~A. Markov.
\newblock Probability theory (4th ed.).
\newblock {\em Moscow (in Russian)}, 1924.

\bibitem{Mich}
A.M. Michalak and Peter~K. Kitanidis.
\newblock Macroscopic behavior and random-walk particle tracking of kinetically
  sorbing solutes.
\newblock {\em Water Resources Research}, 36(8):2133--2146, 2000.

\bibitem{Omey}
E.~Omey, J.~Santos, and S.~Van~Gulck.
\newblock A {M}arkov-binomial distribution.
\newblock {\em Appl. Anal. Discrete Math.}, 2(1):38--50, 2008.

\bibitem{Viveros}
R.~Viveros, K.~Balasubramanian, and N.~Balakrishnan.
\newblock Binomial and negative binomial analogues under correlated {B}ernoulli
  trials.
\newblock {\em The American Statistician}, 48:243--247, 1994.

\bibitem{MR633240}
Y.~H. Wang.
\newblock On the limit of the {M}arkov binomial distribution.
\newblock {\em J. Appl. Probab.}, 18(4):937--942, 1981.

\bibitem{LOGCONCAVITY}
Yi~Wang and Yeong-Nan Yeh.
\newblock Log-concavity and {LC}-positivity.
\newblock {\em J. Combin. Theory Ser. A}, 114(2):195--210, 2007.

\end{thebibliography}
\bibliographystyle{apt}

\end{document}